\documentclass[11pt]{amsart}
\usepackage{amsmath}
\usepackage{amssymb}
\usepackage{amscd}
\usepackage{color}
\usepackage{mathtools}
\def\NZQ{\mathbb}               % the font for N,Z,Q,R,C
\def\NN{{\NZQ N}}
\def\QQ{{\NZQ Q}}
\def\ZZ{{\NZQ Z}}
\def\RR{{\NZQ R}}
\def\CC{{\NZQ C}}

\def\mfp{\mathfrak p}
\def\mfa{\mathfrak a}
\def\Spec{\operatorname{Spec}}
\def \mm{{m_R}}
\def\Ass{\operatorname{Ass}}
\def\MinAss{\operatorname{MinAss}}
\def\height{\operatorname{height}}

\newtheorem{Theorem}{Theorem}[section]
\newtheorem{Lemma}[Theorem]{Lemma}
\newtheorem{Corollary}[Theorem]{Corollary}
\newtheorem{Proposition}[Theorem]{Proposition}
\newtheorem{Remark}[Theorem]{Remark}

\newtheorem{Example}[Theorem]{Example}

\let\epsilon\varepsilon
\let\phi=\varphi
\let\kappa=\varkappa

\begin{document}

\title{Epsilon multiplicity and analytic spread of filtrations}

\author{Steven Dale Cutkosky}

\thanks{The first author was partially supported by NSF grant DMS-2054394}

\address{Steven Dale Cutkosky, Department of Mathematics,
University of Missouri, Columbia, MO 65211, USA}
\email{cutkoskys@missouri.edu}

\author{Parangama Sarkar}
\keywords{filtration, epsilon multiplicity, analytic spread, divisorial filtration}
\subjclass[2000]{primary 13H15, 13A18, 13A02} 

\thanks{The second author was partially supported by SERB POWER Grant with Grant No. SPG/2021/002423.}
\begin{abstract} 
We extend the epsilon multiplicity of ideals defined by Ulrich and Validashti  to epsilon multiplicity of filtrations, and show that under mild assumptions this multiplicity exists as a limit. We show that in rather general rings, the epsilon multiplicity of a $\QQ$-divisorial filtration is positive if and only if the analytic spread of the filtration is maximal (equal to the dimension of the ring). The condition that filtrations $\mathcal J\subset \mathcal I$  have the same epsilon multiplicity is considered, and we find conditions ensuring that the filtrations have the same integral closure.
\end{abstract}
\maketitle

\section{Introduction}
The epsilon multiplicity of an ideal $I$ in a local ring $R$ with maximal ideal $m_R$ is defined in \cite{UV} to be 
$$
\epsilon(I)=d!\limsup_n\frac{\lambda_R(H^0_{m_R}(R/I^n))}{n^d}.
$$
It is shown there, by comparison with the $j$-multiplicity,  that $\epsilon(I)$ is always a real number. It is shown in \cite{CHST} that $\epsilon(I)$ can be an irrational number. 
In a local ring $R$, we have that 
$$
H^0_{m_R}(R/I)=I:m_R^{\infty}/I.
$$

For ideals in analytically unramified local rings, the epsilon multiplicity is a limit. 

\begin{Theorem}( \cite[Corollary 6.3]{C1})\label{TheoremI1} Suppose that $I$ is an ideal in an analytically unramified local ring $R$. Then $\epsilon(I)$ is actually a limit, 
$$
\epsilon(I)
=d!\lim_{n\rightarrow \infty}\frac{\lambda_R(I^n:m_R^{\infty}/I^n)}{n^d}.
$$
\end{Theorem}

In this paper, we extend epsilon multiplicity to filtrations and obtain some results generalizing theorems about epsilon multiplicities for ideals. 

We extend the definition of epsilon multiplicity to  filtrations by defining the epsilon multiplicity of a filtration $\mathcal I=\{I_n\}$ of ideals in $R$ to be
\begin{equation}\label{epdef}
\epsilon(\mathcal I)=d!\limsup_n\frac{\lambda_R(H^0_{m_R}(R/I_n))}{n^d}.
\end{equation}

Let $\mathcal I=\{I_n\}$ be a filtration on a local ring $R$ and $c\in \ZZ_{>0}$. We will say that $\mathcal I$ satisfies property $A(c)$ if 
$$
(I_n:m_R^{\infty})\cap m_R^{cn}=I_n\cap m_R^{cn}  \mbox{ for all $n\in \NN$.}
$$

As a consequence of  \cite[Theorem 6.1]{C1}, we have the following theorem.

\begin{Theorem}\label{ThmE1} Let $R$ be an analytically unramified  local ring of dimension $d$, and $\mathcal I=\{I_n\}$ be a filtration on $R$ which satisfies the property $A(c)$ for some $c\in \ZZ_{>0}$. Then 
\begin{equation}\label{eqI6}
\epsilon(\mathcal I)=d!\lim_{n\rightarrow \infty}\frac{\lambda_R(I_n:m_R^{\infty}/I_n)}{n^d}
\end{equation}
is a real number. 
That is, the epsilon multiplicity of $\mathcal I$ exists as a limit.
\end{Theorem}

%In particular, taking $\mathcal I$ to be the $I$-adic filtration of an ideal $I$ in an unramified local ring, we recover Theorem \ref{TheoremI1}, showing that the epsilon multiplicity of an ideal in an unramified local ring exists as a limit. 

We deduce that the epsilon multiplicity exists for many naturally occurring filtrations.  Discrete valued filtrations are defined in Section \ref{notation}.
 \begin{Theorem}\label{ep}
 	Let $R$ be a  Noetherian local domain, $\mathcal J$  a discrete valued filtration and $\mfp$ any prime ideal in $R$. Then the filtration $\mathcal J_\mfp$ satisfies $A(c_\mfp)$ for some $c_\mfp\in\ZZ_{>0}$. 
 	
 	Moreover, if $R_\mfp$ is an  analytically unramified local domain then $\epsilon(\mathcal J_\mfp)$ exists as a limit. In particular, epsilon multiplicity exists as a limit for  discrete valued filtrations of an analytically unramified local domain.
 	\end{Theorem}

We give examples of filtrations $\mathcal I$ for which the epsilon multiplicity is finite, but the epsilon multiplicity does not exist as a limit and filtrations $\mathcal I$ for which the epsilon multiplicity is infinite. 
 These filtrations necessarily do not satisfy $A(c)$ for any $c$.

The following theorem about epsilon multiplicity of ideals follows from results of Ulrich and Validashti.

\begin{Theorem}\label{TheoremI2} Suppose that $R$ is a universally catenary Noetherian local ring of dimension $d$  and $I$ is an ideal of $R$. Then the analytic spread
$\ell(I)=d$ if and only if $\epsilon(I)>0$.
\end{Theorem}

The proof that $\ell(I)<\dim R$ implies $\epsilon(I)=0$ follows from the theory of $j$-multiplicity (\cite{AM} and \cite{FOV}) as explained in  the inequality on the second line of page 97 of \cite{UV} and Remark 4.2 in \cite{UV1}. This part of the proof is valid for an arbitrary local ring.  By Theorem 4.4 \cite{UV}, if  $R$ is an equidimensional, universally catenary Noetherian local ring  and $\ell(I)=d$, then   $\epsilon(I)>0$.

%The following corollary is immediate  from Theorem \ref{Theorem2} and \cite{MA} or Theorem 5.4.6 \cite{HS}.

%\begin{Corollary}\label{CorI3} Suppose that $R$ is a universally catenary Noetherian local ring of dimension $d$ and $I$ is an ideal of $R$. Let $\nu_1,\ldots,,\nu_r$ be the Rees valuations of $R$. Then $\epsilon(I)>0$ if and only if the center of at least one of the $\nu_i$ on $R$ is the maximal ideal.
%\end{Corollary}

%We now state our generalization of  Theorem \ref{PropE6} to filtrations. Since we are concerned with filtrations whose Rees algebra may not be Noetherian, and thus are far from being generated  in degree 1, the generalization of 
 % the  integral closure of an ideal to filtrations is the  filtration ${\rm IC}(\mathcal I)$ which is defined to be the filtration such that
%$R[{\rm IC}(\mathcal I)]$ is the integral closure of $R[\mathcal I]$.

%\begin{Theorem}\label{TheoremE7} Suppose that $R$ is a universally Nagata analytically unramified local ring and $\mathcal I$ and $\mathcal J$ are filtrations of $R$ such that $\mathcal J$ satisfies $A(c)$ for some $c$ and we have inclusions
%$\mathcal I\subset \mathcal J\subset {\rm IC}(\mathcal I)$. Then $\epsilon(\mathcal J)=\epsilon(\mathcal I)$.
%\end{Theorem}

%\begin{Corollary}\label{CorE9} Suppose that $\mathcal I$ is a bounded filtration of a universally Nagata analytically unramified local ring, so that $IC(\mathcal I)$ is a divisorial filtration $\mathcal I(D)$. Then 
%$$
%\epsilon(\mathcal I)=\epsilon(\mathcal I(D)).
%$$
%\end{Corollary}

 We obtain the following generalization of Theorem \ref{TheoremI2} to filtrations.  Divisorial filtrations and the analytic spread $\ell(\mathcal I)$ of a filtration are defined in Section \ref{notation}.

\begin{Theorem}\label{TheoremI4} Let $R$ be a $d$-dimensional excellent normal local domain of equicharacteristic zero, or an arbitrary excellent local domain of dimension $\le 3$.  Let $\mathcal I$ be a 
$\QQ$-divisorial filtration of $R$. Then 
 $\epsilon(\mathcal I)>0$ if and only if 
the analytic spread $\ell(\mathcal I)=d$.
\end{Theorem}

 Theorem \ref{TheoremI4} is not true for general filtrations. Example \ref{Example7}  gives an example of an $\RR$-divisorial filtration $\mathcal I$ in a regular local ring such that $\epsilon(\mathcal I)>0$ but $\ell(\mathcal I)<\dim R$.

Theorem \ref{TheoremI4} is true for any class of excellent local domains for which resolution of singularities is true. Resolution of singularities is true for reduced finite type schemes over an excellent equicharacteristic zero local ring by \cite{H} and for  reduced finite type schemes over an excellent  local ring of dimension $\le 3$ by \cite{L} and  \cite{CJS} (in dimension two) and by \cite{CP} (in dimension 3). The implication $\epsilon(\mathcal I)>0$ implies $\ell(\mathcal I)=d$  follows from the following theorem.

\begin{Theorem}\label{Theorem2}( \cite[Theorem 1.4]{C2}) Let $R$ be an excellent local domain of equicharacteristic 0, or  of dimension $\le 3$. Let $\mathcal I=\{I_n\}$ be a $\QQ$-divisorial filtration on $R$. Then the following are equivalent.
\begin{enumerate}
\item[1)] The analytic spread of $\mathcal I$ is $\ell(\mathcal I)=\dim R$.
\item[2)] There exists $n_0\in \ZZ_{>0}$ such that $m_R\in \mbox{Ass}(R/I_n)$ if $n\ge n_0$.
\item[3)] $m_R\in \mbox{Ass}(R/I_{m_0})$ for some $m_0\in \ZZ_{>0}$.
\end{enumerate}
\end{Theorem}

We obtain the following corollary to Theorems \ref{TheoremI4} and \ref{Theorem2}.

 \begin{Corollary}\label{CorI10} Let $R$ be a $d$-dimensional excellent normal local domain of equicharacteristic zero, or an arbitrary excellent local domain of dimension $\le 3$.  Let $\mathcal I$ be a 
$\QQ$-divisorial filtration of $R$. Then $\epsilon(\mathcal I)>0$ if and only if for every representation of $\mathcal I=\{I_n\}$ as a rational divisorial filtration
$I_n=I(\nu_1)_{na_1}\cap \cdots \cap I(\nu_r)_{na_r}$ for $n\ge 0$,  the maximal ideal $m_R$ is the center $m_{\nu_i}\cap R$ of  at least one of the  $\nu_i$.
\end{Corollary}

%%%%%%%%%%%%%%%

In the final section we consider epsilon multiplicity under  integral closure of filtrations. 
The following theorem about epsilon multiplicities of ideals is a corollary of  Theorem 2.3 \cite{UV}.

\begin{Theorem}\label{PropE6}
 Suppose that $R$ is a universally catenary local ring and $J\subset I$ are ideals in $R$. Then $\epsilon(I_{\mathfrak p})=\epsilon(J_{\mathfrak p})$ for all $\mathfrak p\in\Spec R$ if and only if $\overline J=\overline I$.
\end{Theorem}

%We also have the following statement.

%\begin{Proposition}(Jeffries, Monta\~no and Varbaro, Proposition 2.1 (b) \cite{JMV})\label{PropI5} Let $R$ be a $d$-dimensional analytically unramified local ring and $I$ be an ideal of $R$. Then $\epsilon(I)=\epsilon(\{\overline{I^n}\})$.
%\end{Proposition}

%We consider the condition on filtrations $\mathcal I\subset \mathcal J$  that their integral closures $\overline{R[\mathcal I]}=\overline{R[\mathcal J]}$ are the same and the condition that they have the same epsilon multiplicity.
%The Rees algebra $R[\mathcal I]$ {\color{blue} of a filtration} is defined in Section \ref{notation}.

%{\color{red}{\color{blue} This part is already proved in MULTIPLICITIES OF CLASSICAL VARIETIES JACK JEFFRIES, JONATHAN MONTANO, AND MATTEO VARBARO, Proposition 2.1 (b)}
%To begin with, state the following simple result which follows from Proposition \ref{PropE6} (Proposition \ref{PropI5} is proven in Section \ref{SecEF}). 

%\begin{Proposition}\label{PropI5} Let $R$ be a $d$-dimensional analytically unramified local ring and $I$ be an ideal of $R$. Then $\epsilon(I)=\epsilon(\{\overline{I^n}\})$.
%\end{Proposition}
%}

Example \ref{Example6} shows that there are filtrations $\mathcal J\subset\mathcal I$ such that $\epsilon(\mathcal J_{\mathfrak p})=\epsilon(\mathcal I_{\mathfrak p})$ for all $\mathfrak p\in \Spec R$ but $\overline{R[\mathcal J]}\ne \overline{R[\mathcal I]}$ and Example \ref{Example 0} shows that there are filtrations $\mathcal J\subset \mathcal I$  such that $\overline{R[\mathcal I]}=\overline {R[\mathcal J]}$ but $\epsilon(\mathcal I)\neq\epsilon(\mathcal J)$. Thus Theorem \ref{PropE6} and Theorem \ref{equality} (stated below), do not extend to arbitrary filtrations.

We show that for some naturally occurring filtrations equality of epsilon multiplicity is equivalent to the integral closures of their Rees algebras being the same. Filtrations which are $s$-divisorial or $s$-bounded are defined in Section \ref{closure}.

\begin{Theorem}\label{equality}
	Let $(R,\mm)$ be an excellent local domain and $\mathcal I$ be a filtration of ideals in $R$. Then the following hold.
	\begin{enumerate}
		\item[{(1)}] Suppose $\mathcal J$ is an $s$-divisorial filtration such 
		that $\mathcal J\subset \mathcal I$. Then the following are equivalent.
		\begin{enumerate}
			\item[{(i)}] $\epsilon(\mathcal I_\mfp)=\epsilon(\mathcal J_\mfp)$ for all $\mfp\in\Spec R$. 
			\item[{(ii)}] $\epsilon(\mathcal I_\mfp)=\epsilon(\mathcal J_\mfp)$ for all $\mfp\in\Spec R$ such that $\ell(\mathcal J_\mfp)=\dim R_\mfp$. 
			\item[{(iii)}] $\epsilon(\mathcal I_\mfp)=\epsilon(\mathcal J_\mfp)$ for all $\mfp\in\Spec R$ with $\dim R/\mfp=s$.
			\item[{(iv)}] $\mathcal I=\mathcal J$.
			\item[{(v)}] $\overline{R[\mathcal I]}=\overline {R[\mathcal J]}$.
		\end{enumerate}
		\item[{(2)}] Suppose $\mathcal J$
		is a bounded $s$-filtration such 
		that $\mathcal J\subset \mathcal I$. Then $\overline{R[\mathcal I]}=\overline {R[\mathcal J]}$  if and only if $\epsilon(\mathcal I_\mfp)=\epsilon(\mathcal J_\mfp)$ for all $\mfp\in\Spec R$ with $\dim R/\mfp=s$.
	\end{enumerate}
\end{Theorem}

\section{Notation}\label{notation}
Let $(R, m_R)$  be a Noetherian local ring of dimension $d$. 
%We will often write $m_R=\mathfrak m$.
A descending chain
$$
R=I_0\supset I_1\supset I_2\supset\cdots
$$
of ideals in $R$ is called a filtration if $I_mI_n\subset I_{m+n}$ for all $m,n\in \NN$. For a filtration $\mathcal I=\{I_n\}_{n\in \NN}$, we define the Rees algebra of $\mathcal I$ to be the graded $R$-algebra $R[\mathcal I]=\sum_{n\ge 0}I_nt^n$. This generalizes the Rees algebra $R[I]=\sum_{n\ge 0}I^nt^n$ of an ideal $I$ in $R$. A filtration $\mathcal I$ is called a Noetherian filtration if $R[\mathcal I]$ is a finitely-generated $R$-algebra. Otherwise it is called a non-Noetherian filtration. If $I\subset R$ is an ideal, then $V(I):=\{\mathfrak p\in \mbox{Spec}(R)\mid I\subset \mathfrak p\}$. For a filtration $\mathcal I=\{I_n\}$, we have \cite[Lemma 3.1]{CS},
$$
V(I_1)=V(I_n)\mbox{ and }\dim R/I_1=\dim R/I_n\mbox{ for all }n\ge 1.
$$
 For a filtration $\mathcal I=\{I_n\}$, by $\mathcal I_\mfp$, we denote the filtration $\{I_nR_\mfp\}$
 for any $\mfp\in\Spec R$. For any two filtrations $\mathcal I=\{I_n\}, \mathcal J=\{J_n\}$, by $\mathcal J\subset \mathcal I$, we mean $J_n\subset I_n$ for all $n\geq 0$ and by $\mathcal J= \mathcal I$, we mean $J_n= I_n$ for all $n\geq 0$. 

The analytic spread of a filtration $\mathcal I$ of a local ring is defined to be $\ell(\mathcal I)=\dim R[\mathcal I]/m_RR[\mathcal I]$. It is shown in \cite[Lemma 3.6]{CSas} that $\ell(\mathcal I)\le \dim R$. Even for symbolic filtrations we can have that $\ell(\mathcal I)=0$ (\cite[Theorem 1.11]{CSas}). A different definition of analytic spread is given in \cite{BS}, \cite{HKTT} and \cite{DM}.

Now let $R$ be a Noetherian local domain of dimension $d$ with quotient field $K$. Let $\nu$ be a discrete valuation of $K$ with valuation ring $\mathcal O_{\nu}$ and maximal ideal $m_{\nu}$. Suppose that $R\subset \mathcal O_{\nu}$. Then for $n\in \NN$, define valuation ideals 
$$
I(\nu)_n=\{f\in R\mid \nu(f)\ge n\}=m_{\nu}^n\cap R.
$$
A discrete valued filtration of $R$ is a filtration $\mathcal I=\{I_n\}$ such that there exist discrete valuations $\nu_1,\ldots,\nu_r$ and $a_1,\ldots, a_r\in \RR_{>0}$ such that for all $m\in \NN$,
$$
I_m=I(\nu_1)_{\lceil ma_1\rceil}\cap\cdots\cap I(\nu_r)_{\lceil ma_r\rceil}
$$
 where $\lceil x\rceil$ denotes the round up of a real number $x$.
A discrete valued filtration is called integral (rational) if $a_i\in \ZZ_{>0}$ for all $i$ ($a_i\in \QQ_{>0}$) for all $i$. 

A divisorial filtration of $R$ is a  valuation $\nu$ of the quotient field $K$ of $R$ such that $\nu$ is positive on $R$ and letting $P=R\cap m_R$, the transcendence degree of the residue field of the valuation ring of $\nu$ over the residue field of $R/P$ is ${\rm ht}(P)-1$. If $R$ is excellent, then a valuation $\nu$ of $K$ is a divisorial valuation if and only if the valuation ring $\mathcal O_{\nu}$ of $\nu$ is essentially of finite type over $R$. A divisorial valuation is a discrete valuation.

A divisorial filtration of $R$ is a discrete valued filtration 
$$
\{I_m=I(\nu_1)_{\lceil ma_1\rceil}\cap\cdots\cap I(\nu_r)_{\lceil ma_r\rceil}\}
$$
where all the discrete valuations $\nu_1,\ldots,\nu_r$ are divisorial valuations.  A divisorial filtration is called integral (rational) if $a_i\in \ZZ_{>0}$ for all $i$ ($a_i\in \QQ_{>0}$) for all $i$.

It is shown in \cite[Lemma 3.6]{CS} that for a filtration $\mathcal I=\{I_n\}$, the integral closure of $R[\mathcal I]$ in $R[t]$ is
$$
\overline{R[\mathcal I]}=\sum_{m\ge 0}J_mt^m
$$
where $\{J_m\}$ is the filtration
$$
J_m=\{f\in R\mid f^r\in \overline{I_{rm}}\mbox{ for some }r>0\}.
$$

The following result follows using the same lines of proof of \cite[Lemma 5.7]{C}
\begin{Lemma}\label{integrallyclosed}
If $\mathcal I$ is a discrete valued filtration then $\overline{R[\mathcal I]}=R[\mathcal I]$.
\end{Lemma}

If $I\subset R$ is an ideal in $R$, we define
$I^{\rm sat}=I:m_R^{\infty}=\cup_{n=1}^{\infty}I:m_R^n$.

\section{Epsilon multiplicity of filtrations} \label{SecEF} 

%{\color{blue} {\bf Proposition \ref{PropE10} is not required as it follows from Theorem \ref{ep}.}}

%{\color{red}Theorem \ref{CorI8} follows from Theorem \ref{ThmE1} and Proposition \ref{PropE10}.

%\begin{Proposition}\label{PropE10} Let $R$ be a local domain and $\mathcal I$ be a discrete valued filtration of $R$. Then $\mathcal I$ satisfies $A(c)$ for some $c$.
%\end{Proposition}

%\begin{proof} Let $\mathcal I=\{I_n\}$. There exist discrete valuations $\nu_1,\ldots,\nu_s$ dominating $R$ and $a_1,\ldots,a_s\in \RR_{>0}$ such that $I_n=I(\nu_1)_{na_1}\cap\cdots\cap I(\nu_r)_{na_s}$. We may index the $\nu_i$ so that  $\nu_i$ has center $m_R$ on $R$ for $1\le i\le r$ and $\nu_i$ has higher dimensional center for $i>r$. Then
%$I_n:m_R^{\infty}=I(\nu_1)_{na_{r+1}}\cap \cdots\cap I(\nu_s)_{na_s}$ for $n\in \NN$. Let $c=\lceil \max{\{a_1,\ldots, a_r\}}\rceil$. Since $m_R\subset I(\nu_i)_1$ for $1\le i\le r$, we have that $m_R^{cn}\subset I(\nu_i)_{na_i}$ for $1\le i\le r$. Thus
%$(I_n:m_R^{\infty})\cap m_R^{nc}=I_n\cap m_R^{nc}$ for all $n$.
%\end{proof}

%Theorem \ref{CorI8} follows from Theorem \ref{ThmE1} and  Proposition \ref{PropE10}.

%We give a more general version of Proposition \ref{PropE10}.
%}

The property $A(c)$ of a filtration is defined in the introduction. We obtain the following theorem, showing that epsilon multiplicity, defined in (\ref{epdef}), exits  as a limit for discrete valued filtrations. 

\begin{Theorem}\label{ep1}
	Let $R$ be a  Noetherian local domain, $\mathcal J$  a discrete valued filtration and $\mfp$ any prime ideal in $R$. Then the filtration $\mathcal J_\mfp$ satisfies $A(c_\mfp)$ for some $c_\mfp\in\ZZ_{>0}$. 
	
	Moreover, if $R_\mfp$ is an  analytically unramified local domain then $\epsilon(\mathcal J_\mfp)$ exists as a limit. In particular, epsilon multiplicity exists as a limit for  discrete valued filtrations of an analytically unramified local domain.
%	
%	 for all $\mfp\in\Spec R$ there exists $c_\mfp\in\ZZ_{>0}$ such that for all $n\geq 1$,
%	$$(J_nR_\mfp)^{sat}\bigcap\mfp^{nc_\mfp} R_{\mfp}=J_nR_\mfp\bigcap\mfp^{nc_\mfp}R_{\mfp}.$$
%	Moreover, if $R$ is an  analytically unramified local domain of dimension $d>0$ then the epsilon multiplicity   $$\epsilon(\mathcal J_\mfp)=\lim\limits_{n\to \infty}\frac{dim R_\mfp!}{n^{\dim R_\mfp}}\ell_{R_\mfp}(H_{\mfp R_\mfp}^0(R_\mfp/J_nR_\mfp))$$ exists for all $\mfp\in\Spec R$. 
	
\end{Theorem}	

\begin{proof}
	Let $\mathcal J=\{J_n=I(\nu_1)_{\lceil na_1\rceil }\cap\cdots\cap I(\nu_r)_{\lceil na_r\rceil}\}$. Reindexing the $\nu_i$ as $\nu_{ij}$, let
	$\mfp_i=\mathfrak m_{\nu_{ij}}\cap R$ for $i=1,\ldots,l$ and $j=1,\ldots,k_i$  be the distinct centers of the discrete valuations $\nu_1,\ldots,\nu_r$ where $k_1+\cdots+k_l=r$.
	
	Let $\mfp\in\Spec R$. If $\mfp_i\nsubseteq\mfp$ for all $i=1,\ldots,l$ then $\mfp\notin V(J_1)$ and hence $(J_nR_\mfp)^{sat}=J_nR_\mfp$. We take $c_{\mfp}=1$ in this case. 
	
	Suppose $\mfp_i\subseteq\mfp$ for some $i\in\{1,\ldots,l\}$. Without loss of generality, we assume $\mfp_i\subseteq \mfp$ for all $1\leq i\leq t$ and  $\mfp_i\nsubseteq\mfp$ for all $t+1\leq i\leq l$.
	
	Case 1: Suppose $\mfp_i\subsetneq \mfp$ for all $1\leq i\leq t$. Then $$(J_nR_\mfp)^{sat}=\bigcap\limits_{\substack{\mathfrak m_{\nu_{ij}}\cap R=\mfp_i\\1\leq i\leq t, 1\leq j\leq k_i}}I(\nu_{ij})_{\lceil na_{ij}\rceil }R_\mfp=J_nR_\mfp.$$ Thus we take $c_{\mfp}=1$. 
	
	Case 2: Suppose $\mfp_i= \mfp$ for some $i\in\{1,\ldots,t\}$. Without loss of generality, we assume $i=1$. Then $$(J_nR_\mfp)^{sat}=\bigcap\limits_{\substack{\mathfrak m_{\nu_{ij}}\cap R=\mfp_i\\2\leq i\leq t, 1\leq j\leq k_i}}I(\nu_{ij})_{\lceil na_{ij}\rceil }R_\mfp.$$
	
	Now $\mfp_1^{\lceil na_{1j}\rceil}\subset I(\nu_{i1})_{\lceil na_{1j}\rceil}$ for all $1\leq j\leq k_1$. Let $c_\mfp=\max\{\lceil a_{11}\rceil,\ldots,\lceil a_{1k_1}\rceil\}$. Then 
	\begin{eqnarray*}
		(J_nR_\mfp)^{sat}\bigcap\mfp^{nc_\mfp}R_{\mfp}
		&=&\bigcap\limits_{\substack{\mathfrak m_{\nu_{ij}}\cap R=\mfp_i\\2\leq i\leq t, 1\leq j\leq k_i}}I(\nu_{ij})_{\lceil na_{ij}\rceil }R_\mfp\cap\mfp^{nc_\mfp}R_{\mfp}
		\\&\subset& \bigcap\limits_{\substack{\mathfrak m_{\nu_{ij}}\cap R
				=\mfp_i\\2\leq i\leq t, 1\leq j\leq k_i}}I(\nu_{ij})_{\lceil na_{ij}\rceil }R_\mfp\bigcap\big(\bigcap\limits_{1\leq j\leq k_1}I(\nu_{1j})_{\lceil na_{1j}\rceil }R_\mfp\big)\bigcap\mfp^{nc_\mfp}R_{\mfp}
		\\&\subset& J_nR_\mfp\bigcap\mfp^{nc_\mfp}R_{\mfp}.
	\end{eqnarray*}
	Thus if $R_\mfp$ is an  analytically unramified local domain then by Theorem \ref{ThmE1}, $\epsilon(\mathcal J_\mfp)$ exists. 
\end{proof}	

\begin{Example}
Let $R=k[x,y]_{(x,y)}$ where $k$ is a field and $\mathcal I=\{I_n=(x)^{{\lceil{n\pi}\rceil}}\cap (x,y)^{\lceil{2n\pi}\rceil}\}$ be a non-Noetherian discrete valued filtration in $R$.	Note that $V(I_n)=\{P=(x), \mm=(x,y)\}$. Since $\mathcal I_\mfp=\{I_nR_P=(x)^{{\lceil{n\pi}\rceil}}R_P\}$, we have $\epsilon(\mathcal I_P)=\pi$. Now  $I_n^{sat}=(x)^{{\lceil{n\pi}\rceil}}$. Therefore \begin{eqnarray*}
\epsilon(\mathcal I)&=&2!\lim\limits_{n\to\infty}\lambda_R\Big((x)^{{\lceil{n\pi}\rceil}}/(x)^{{\lceil{n\pi}\rceil}}\cap (x,y)^{\lceil{2n\pi}\rceil}\Big)/n^2\\&=&2!\lim\limits_{n\to\infty}\frac{
	({\lceil{2n\pi}\rceil}-{\lceil{n\pi}\rceil})({\lceil{2n\pi}\rceil}-{\lceil{n\pi}\rceil}+1)}{2n^2}\\&=& \pi^2.
\end{eqnarray*} 
\end{Example}

%{\color{red}{\color{blue} This part is already proved in MULTIPLICITIES OF CLASSICAL VARIETIES JACK JEFFRIES, JONATHAN MONTANO, AND MATTEO VARBARO, Proposition 2.1 (b)}
	
	%We now prove Proposition \ref{PropI5}. Since $R$ is analytically unramified, $R[\{\overline{I^n}\}]=\overline {R[I]}$ is a finite $R[I]$-module, by \cite[Corollary 9.2.1]{HS}. Thus there exists $m>0$ such that 
%$\oplus_{n\ge 0}\overline{I^{mn}}$ is generated in degree 1.  Now
%$\epsilon(I^m)=\epsilon)\overline(\overline{I^m})$ by \cite[Theorem 2.3]{UV}, and $\epsilon(\overline{I^m})=\epsilon(\{\overline{I^{mn}})$, so
%$$
%\begin{array}{lll}
%\epsilon(I)&=&\lim_{n\rightarrow\infty}d!\frac{\lambda_R((I^n:m_R^{\infty})/I^n)}{n^d}
%= \frac{1}{m^d}\lim_{n\rightarrow\infty}d!\frac{\lambda_R((I^{mn}:m_R^{\infty})/I^{mn})}{n^d}\\
%&=&  \frac{1}{m^d}\lim_{n\rightarrow\infty}d!\frac{\lambda_R((\overline{I^{mn}}:m_R^{\infty})/\overline{I^{mn}})}{n^d}
%=\lim_{n\rightarrow\infty}d!\frac{\lambda_R((\overline{I^n}:m_R^{\infty})/\overline{I^n})}{n^d}=\epsilon(\{\overline{I^n}\}).
%\end{array}
%$$}

Now we introduce a family of examples, which will illustrate the essential role  of the $A(c)$ condition in the existence of epsilon multiplicity as a limit.

Let $\tau:\ZZ_{>0}\rightarrow \ZZ_{>0}$ be any function such that $\tau(n+1)\ge \tau(n)$ for all $n$. We will restrict to $\tau$ satisfying this condition in this analysis. 

Then defining $I_n$ to be the ideal $I_n=(x^2,xy^{\tau(n)})$ in the  local ring $R=k[x,y]_{(x,y)}$ over a field $k$, 
we have that $\mathcal I_{\tau}=\{I_n\}$ is a filtration. We have that $I_n:m_R^{\infty}=(x)$ and 
$(I_n:m_R^{\infty})/I_n\cong R/(x,y^{\tau(n)})$ as an $R$-module. Thus $\lambda_R(I_n:m_R^{\infty}/I_n)=\tau(n)$.

Note that for any  $\tau$ and $f\in I_n$ for $n\geq 1$, we have $f^2\in\mm I_{2n}$. Therefore by \cite[Lemma 3.8]{CSas}, $\ell(\mathcal I_{\tau})=0$. We conclude that for all $\tau$,
the epsilon multiplicity is 
$$
\epsilon(\mathcal I_{\tau})=2\limsup 
\frac{\tau(n)}{n^2}
$$
while the analytic spread is
$$
\ell(\mathcal I_{\tau})=0.
$$
The following examples  show that we cannot expect the epsilon multiplicity of a filtration $\mathcal I$ to exist as a limit if the condition $A(c)$ is not satisfied for any $c$.

\begin{Example}\label{Example1}
Let $\sigma:\ZZ_{>0}\rightarrow \ZZ_{>0}$ be the function defined in (22) of Section 5 of \cite{C1}. We have that $\sigma(n+1)\ge \sigma(n)$ for all $n$, $1\le \sigma(n)\le \frac{n}{2}$ for all $n$ and $\lim_{n\rightarrow\infty}\frac{\sigma(n)}{n}$ does not exist, even when $n$ is constrained to lie in any arithmetic sequence. We further have that $\limsup \frac{\sigma(n)}{n}=\frac{1}{2}$. Let $\tau(n)=n\sigma(n)$. Then 
$$
\lim_{n\rightarrow \infty} \frac{\lambda_R(I_n:m_R^{\infty}/I_n)}{n^2}=\lim_{n\rightarrow\infty}\frac{\sigma(n)}{n}
$$
 does not exist. In this example, we have that $\epsilon(\mathcal I_{\tau})=\frac{1}{2}<\infty$, so that $\epsilon(\mathcal I_{\tau})$ is positive but $\ell(\mathcal I_{\tau})\ne \dim R$.
 \end{Example}

 \begin{Example}\label{Example2}
Let $\tau(n)$  be any increasing function  such  that $\limsup \frac{\tau(n)}{n^2}=\infty$ (such as $\tau(n)=n^3$). We obtain that $\epsilon(\mathcal I_{\tau})=\infty$. In particular, $\epsilon(\mathcal I_{\tau})$ is not finite.
 \end{Example}
 
 \begin{Example}\label{Example3} Suppose that $\mathcal I_{\tau}$ satisfies condition $A(c)$ for some $c$. Then $(I_n:m_R^{\infty})\cap m_R^{cn}=I_n\cap m_R^{cn}$ for all $n$ so that  $\tau(n)\le cn$ for all $n$. Thus $\frac{\lambda_R(I_n:m_R^{\infty}/I_n)}{n^2}=\frac{\tau(n)}{n^2}\le \frac{c}{n}$ for all $n$. Thus $\epsilon(\mathcal I_{\tau})=0$ and this  multiplicity exists as a limit, in agreement with the conclusions of Theorem \ref{ThmE1}. We further have that $\epsilon(\mathcal I_{\tau})=0$ and $\ell(\mathcal I_{\tau})\ne \dim R$.
\end{Example}

\section{Analytic spread and epsilon multiplicity}

In this section, we prove Theorem \ref{TheoremI4} from the introduction, which we restate here for the convenience of the reader.

\begin{Theorem} Let $R$ be a $d$-dimensional excellent normal local domain of equicharacteristic zero, or an arbitrary excellent local domain of dimension $\le 3$.  Let $\mathcal I$ be a 
$\QQ$-divisorial filtration of $R$. Then 
 $\epsilon(\mathcal I)>0$ if and only if 
the analytic spread $\ell(\mathcal I)=d$.
\end{Theorem}

The following example shows that the conclusions of Theorem \ref{TheoremI2} and Theorem \ref{Theorem2} are false for filtrations.
We make use of 	\cite[Example 1.5]{C2} which shows that the conclusions of Theorem \ref{Theorem2} do not hold for $\RR$-divisorial filtrations. 

\begin{Example}\label{Example7} There exists an $\RR$-divisorial filtration $\mathcal I$ in $R=k[x]_{(x)}$ such that $\epsilon(\mathcal I)>0$ but $\ell(\mathcal I)=0<1=\dim R$.
\end{Example}

The example satisfies $A(4)$. This is the construction of the example. Let $\mathcal I=\{I_n\}$ where $I_n=(x^{\lceil n\pi\rceil})$ in $R=k[x]_{(x)}$.  For fixed $n$, $r\lceil n\pi\rceil>\lceil rn\pi\rceil+1$ for some $r\in \ZZ_{>0}$. Thus $f\in I_n$ implies $f^r\in m_RI_{nr}$ and so by \cite[Lemma 3.8]{CSas},
$$
\ell(\mathcal I)=\dim R[\mathcal I]/m_RR[\mathcal I]=0<1=\dim R.
$$
For all $n$, $I_n:m_R^{\infty}=R$ so 
$$
\epsilon(\mathcal I)=\lim_{n\rightarrow\infty}\frac{\lceil n\pi\rceil}{n}=\pi>0.
$$

Let $R$ be a normal excellent local ring. Let $\mathcal I=\{I_m\}$ where 
 $$
I_m=I(\nu_1)_{ma_1}\cap\cdots\cap I(\nu_s)_{ma_s}
$$ 
 for some divisorial valuations $\nu_1,\ldots,\nu_s$ of $R$ be an $\RR$-divisorial filtration of $R$, with $a_1,\ldots, a_s\in \RR_{>0}$. Then there exists a projective birational morphism $\phi:X\rightarrow \mbox{Spec}(R)$ such that there exist prime divisors $F_1,
 \ldots, F_s$ on $X$ such that $V_{\nu_i}=\mathcal O_{X,F_i}$ for  $1\le i\le s$ (\cite[Remark 6.6 to Lemma 6.5]{CS}). Let $D=a_1F_1+\cdots+a_sF_s$, an effective $\RR$-divisor on $X$ (an effective $\RR$-Weil divisor). Define $\lceil D\rceil=\lceil a_1\rceil F_1+\cdots+\lceil a_s\rceil F_s$, an integral divisor. 
 We have coherent sheaves $\mathcal O_X(-\lceil n D\rceil)$ on $X$ such that 
 \begin{equation}\label{N1}
 \Gamma(X,\mathcal O_X(-\lceil nD\rceil ))=I_n
 \end{equation}
  for $n\in \NN$. If $X$ is nonsingular then $\mathcal O_X(-\lceil nD\rceil)$ is invertible. The formula (\ref{N1}) is independent of choice of $X$. Further, even on a particular $X$, there are generally many different choices of effective $\RR$-divisors $G$ on $X$ such that $\Gamma(X,\mathcal O_X(-\lceil nG\rceil))=I_n$ for all $n\in \NN$. Any choice of a divisor $G$ on such an $X$ for which the formula $\Gamma(X,\mathcal O_X(-\lceil nG\rceil))=I_n$ for all $n\in \NN$ holds will be called a representation of the filtration  $\mathcal I$. 
 
 Given an  $\RR$-divisor $D=a_1F_1+\cdots+a_sF_s$ on $X$ we have a divisorial filtration
 $\mathcal I(D)=\{I(nD)\}$ where 
 $$
 I(nD)=\Gamma(X,\mathcal O_X(-\lceil nD\rceil ))=I(\nu_1)_{\lceil na_1\rceil}\cap\cdots\cap I(\nu_s)_{\lceil na_s\rceil} 
 =I(\nu_1)_{na_1}\cap\cdots\cap I(\nu_s)_{na_s}.
 $$ 
 We write $R[D]=R[\mathcal I(D)]$.

We recall the $\gamma_{\Gamma}$ function defined in \cite[Section 3]{C7}. We make use of  statements and proofs in  \cite[Section 4]{C2} which are based on corresponding statements and proofs for the $\sigma_{\Gamma}$ function on pseudo-effective divisors on a projective nonsingular variety in \cite[Chapter III, Section 1]{Nak}. 

Let $R$ be a normal excellent local ring and  $\pi:X\rightarrow \mbox{Spec}(R)$ be a birational projective morphism such that $X$ is nonsingular. 
Let $G=\sum a_iE_i$ be an effective $\RR$-divisor, and $\Gamma$ be a prime divisor on $X$. 
Let
$$
\mbox{ord}_{\Gamma}(G)=\left\{\begin{array}{ll}
a_i&\mbox{ if }\Gamma=E_i\\
0&\mbox{if }\Gamma\not\subset \mbox{Supp}(D).
\end{array}\right.
$$
For $D$ an $\RR$-divisor, let 
$$
\tau_{\Gamma}(D)=\inf\{\mbox{ord}_{\Gamma}(G)\mid G\ge 0\mbox{ and }G\sim D\},
$$
and define
$$
\gamma_{\Gamma}(D)=\inf\left\{\frac{\tau_{\Gamma}(mD)}{m}\mid m\in \ZZ_{>0}\right\}.
$$
Since $D$ is linearly equivalent to an effective divisor, there can be  only finitely many prime divisors $\Gamma$ such that $\gamma_{\Gamma}(D)>0$.

If $R$ has dimension 2 and $D$ is an integral divisor on $X$ then $\gamma_{\Gamma}(D)$ is a rational number, but there exist examples where $R$ has dimension 3 and integral divisors $D$ on $X$ such that $\gamma_{\Gamma}(D)$
 is an irrational number (\cite[Theorem 4.1]{C3}).

We have that 
$$
\gamma_{\Gamma}(D)=\lim_{m\rightarrow \infty}\frac{\tau_{\Gamma}(mD)}{m}.
$$
Thus if $\alpha\in \QQ_{>0}$, we have that
\begin{equation}\label{eq10}
\gamma_{\Gamma}(\alpha D)=\alpha\gamma_{\Gamma}(D).
\end{equation}
We also have that for $\RR$ divisors $D_1$ and $D_2$,
\begin{equation}\label{eq11}
\gamma_{\Gamma}(D_1+D_2)\le\gamma_{\Gamma}(D_1)+\gamma_{\Gamma}(D_2).
\end{equation}
If $m\in \ZZ_{>0}$ and $G\sim mD$ is an effective $\RR$-divisor, then $\mbox{ord}_{\Gamma}(G)\ge m\gamma_{\Gamma}(D)$.

\vskip .2truein

We now prove Theorem \ref{TheoremI4}. Let notation be as in the statement of Theorem \ref{TheoremI4}.

Suppose that $\epsilon(\mathcal I)>0$. Then $H^0_{m_R}(R/I_n)\ne 0$ for some $n$ which implies that $m_R\in \mbox{Ass}(R/I_n)$. By Theorem \ref{Theorem2}, $\ell(\mathcal I)=d$.

We will now show that $\ell(\mathcal I)=d$ implies $\epsilon(\mathcal I)>0$. The first part of the proof is similar to the first part of the proof of \cite[Theorem 1.4]{C2}.
Let $k=R/m_R$.
There exists by \cite[Lemma 3.1]{C2}, a projective birational morphism $\pi:X\rightarrow \mbox{Spec}(R)$ such that $X$ is nonsingular, all prime exceptional divisors of $\pi$ are nonsingular
 and there exists an effective $\QQ$-divisor $D$ on $X$ such that $\mathcal I=\mathcal I(D)$. Let 
 \begin{equation}\label{eq**}
 D=\sum a_iF_i
\end{equation}
with the $F_i$ distinct prime divisors and $a_i\in \QQ_{\ge 0}$.

By Theorem \ref{Theorem2}, $\ell(\mathcal I)=d$ implies there exists an $m'$ such that $m_R\in \mbox{Ass}(R/I_{m'})$. Thus 
there exists an $F_j$ which contracts to $m_R$ and $m'>0$ such that if $D_1=\sum_{i\ne j}a_iF_i$, then \begin{equation}\label{eq*}
I(m'D)\ne I(m'D_1).
\end{equation}
 Without loss of generality, we may assume that $j=1$ so that $F_j=F_1$. Set $F=F_1$. Since $F\subset \pi^{-1}(m_R)$, $F$ is a projective $k$-variety.
 We have that
 $-D\le -D_1=-D+a_1F$. Suppose that $\gamma_{F}(-D)>0$. Then $\gamma_{F}(-D_1)=\gamma_{F}(-D)+a_1$ by \cite[Lemma 4.3]{C2}. Thus 
$$
\Gamma(X,\mathcal O_X(\lfloor -nD\rfloor))=\Gamma(X,\mathcal O_X(\lfloor -nD_1\rfloor))
$$
 for all $n\in \NN$ by \cite[Lemma 4.1]{C2}, which is a contradiction to (\ref{eq*}). Thus $\gamma_{F}(-D)=0$ and $0\le \gamma_{F}(-D_1)<a_1$. 
 
Let $K$ be an ideal of $R$ such that $\pi:X\rightarrow \mbox{Spec}(R)$ is the blowup of $K$. Let $A$ be the Cartier divisor defined by $K\mathcal O_X=\mathcal O_X(A)$. Expand $A=\sum -b_iF_i$ where we increase the set of $F_i$ if neccessary, so that some $a_i$ and $b_j$ could be zero. $A$ is an anti-effective integral divisor which is ample and all irreducible components of $\pi^{-1}(m_R)$ are in the support of $A$. After possibly replacing $K$ with a power of $K$, so that $A$ is replaced with a multiple of $A$, we may assume that $A-F$ is also ample. We have that
$$
F=-\frac{1}{b_1}A-\sum_{i> 1}\frac{b_i}{b_1}F_i.
$$

We establish the following lemma. For $0<t\in \QQ$, let  $D_t=D-tF$. 

\begin{Lemma}\label{LemmaG1} Suppose that $t<a_1-\gamma_{F}(-D_1)$ is a positive rational number.  Then there exists $m_0>0$ and an effective integral divisor $U$ on $X$ which does not contain $F$ in it's support such that $m_0D_t$ is an integral divisor and
$-m_0D_t\sim U$.
\end{Lemma}

\begin{proof} Let $\lambda\in \QQ_{>0}$ be such that $\lambda+t<a_1-\gamma_{F}(-D_1)$. 
Then $D_t=D_{\lambda+t}+\lambda F$.   Suppose that $\gamma_{F}(-D_t)>0$. Then 
by \cite[Lemma 4.2]{C2} \ and  \cite[Lemma 4.3]{C2},
$$
\begin{array}{lll}
0=\gamma_{F}(-D_1-\gamma_{F}(-D_1)F)&=&\gamma_{F}(-D_t+(a_1-\gamma_{F}(-D_1)-t)F)\\
&=&\gamma_{F}(-D_t)+(a_1-\gamma_{F}(-D_1)-t)>0,
\end{array}
$$
 a contradiction. Thus $\gamma_{F}(-D_t)=0$.

 We have that
$$
-D_t=-D_{\lambda+t}-\lambda F=-D_{\lambda+t}+\frac{\lambda}{b_1}A+\sum_{i> 1}\frac{\lambda b_i}{b_1}F_i.
$$
Now the lemma follows from Lemma 4.4 \cite{C2}, since $\gamma_{F}(-D_{\lambda+t})=0$.
\end{proof}

For the rest of the proof, we fix $t\in \QQ_{>0}$ with $t<a_1-\gamma_{F}(-D_1)$ and allow $s\in \QQ_{>0}$ with $0<s<t$ to vary. 

We have that $-D_s=-D_t+(s-t)F$. By Lemma \ref{LemmaG1}, there exists $m_0\in \ZZ_{>0}$ and an effective integral divisor $U$ whose support does not contain $F$ such that $m_0(-D_t)\sim U$.  Thus
$$
m_0(-D_s)\sim U+m_0(s-t)F\sim U+m_0(t-s)(\sum_{i>1}\frac{b_i}{b_1}F_i)+\frac{m_0(t-s)}{b_1}A.
$$
Since $A$ and $A-F$ are ample, we have that $A-\lambda F$ is ample for $0\le \lambda\le 1$. 

We now impose the further restriction on $s$ that $s\le\frac{t}{2b_1+1}$. This implies that $\frac{sb_1}{t-s}\le \frac{1}{2}$. With this restriction on $s$, we have that
$\frac{t-s}{b_1}A-\lambda F$ is ample for $0\le\lambda\le s$. In fact,
\begin{equation}\label{eqR4}
A-\frac{b_1\lambda}{t-s}F=\frac{1}{2}A+(\frac{1}{2}A-\frac{b_1\lambda}{t-s}F)
=\frac{1}{2}A+\frac{1}{2}(A-\frac{2b_1\lambda}{t-s}F)
\end{equation}
which is the sum of two ample divisors since $\frac{2b_1\lambda}{t-s}\le 1$. Let $H_s=\frac{t-s}{b_1}A$.

There exists $m_1\in \ZZ_{>0}$ such that $m_1\left(\frac{t-s}{b_1}\right)\in \ZZ_{>0}$, $m_1s\in \ZZ_{>0}$ and $m_1D$ is an integral divisor. Let $m_s=m_0m_1$.

 There exists a section 
$$
\tau_s\in \Gamma(X,\mathcal O_X(-m_sD_s-m_sH_s))
$$
 such that the divisor $(\tau_s)$ of $\tau_s$ is $m_1U+m_s\left(\frac{t-s}{b_1}\right)(\sum_{i>1}b_iF_i)$. Since $F$ is not in the support of $(\tau_s)$, $\tau_s$ restricts to a nonzero section $\overline\tau_s$ of $\Gamma(F,\mathcal O_X(-m_sD_s-m_sH_s)\otimes\mathcal O_{nm_ssF})$, inducing  commutative diagrams with exact columns and   rows for $n\in \ZZ_{\ge 0}$
 
\begin{equation}\label{eqR2}
\begin{array}{ccccc}
&&0&&0\\
&&\uparrow&&\uparrow\\
0&\rightarrow &\mathcal O_X(nm_sH_s)\otimes\mathcal O_{nm_ssF}&\stackrel{\overline\tau_s^n}{\rightarrow}&
\mathcal O_X(-nm_sD+nm_ssF)\otimes \mathcal O_{nm_ssF}\\
&&\uparrow&&\uparrow\\
0&\rightarrow&\mathcal O_X(nm_sH_s)&\stackrel{\tau_s^n}{\rightarrow}& \mathcal O_X(-nm_sD+nm_ssF)\\
&&\uparrow&&\uparrow\\
0&\rightarrow &\mathcal O_X(nm_sH_s-nm_ssF)&\rightarrow&\mathcal O_X(-nm_sD)\\
&&\uparrow&&\uparrow\\
&&0&&0
\end{array}
\end{equation}

%\begin{equation}\label{eqR2}
%\begin{array}{ccccccccc}
%0&\rightarrow&\mathcal O_X(-nm_sD)&\rightarrow &\mathcal O_X(-nm_sD+nm_ssF)&\rightarrow&\mathcal O_X(-nm_sD+nm_ssF)\otimes \mathcal O_{nm_ssF}&\rightarrow&0\\
%&&\uparrow&&\uparrow\tau_s^n&&\uparrow \overline\tau_s^n\\
%0&\rightarrow &\mathcal O_X(nm_sH_s-nm_ssF)&\rightarrow &\mathcal O_X(nm_sH_s)&\rightarrow &\mathcal O_X(nm_sH_s)\otimes\mathcal O_{nm_ssF}&\rightarrow&0
%\end{array}
%\end{equation}
%where the vertical arrows are injections.
%We thus have commutative diagrams 
%\begin{equation}\label{eq3}
%\begin{array}{ccc}
%H^0(X,\mathcal O_X(-nm_sD+nm_ssF))&\stackrel{\alpha_n}{\rightarrow}& H^0(nm_ssF,
%\mathcal O_X(-nm_sD+nm_ssF)\otimes\mathcal O_{nm_ssF})\\
%\uparrow \tau_s^n&&\uparrow \overline\tau_s^n\\
%H^0(X,\mathcal O_X(nm_sH_s))&\stackrel{\beta_n}{\rightarrow}& H^0(nm_ssF,
%\mathcal O_X(nm_sH_s)\otimes\mathcal O_{nm_ssF}).
%\end{array}
%\end{equation}

and taking cohomology, we have commutative diagrams with exact columns and rows
$$
\begin{array}{ccccc}
0&\rightarrow &H^0(X,\mathcal O_X(nm_sH_s)\otimes\mathcal O_{nm_ssF})&\stackrel{\overline\tau_s^n}{\rightarrow}&
H^0(X,\mathcal O_X(-nm_sD+nm_ssF)\otimes \mathcal O_{nm_ssF})\\
&&\beta_n\uparrow&&\alpha_n\uparrow\\
0&\rightarrow&H^0(X,\mathcal O_X(nm_sH_s))&\stackrel{\tau_s^n}{\rightarrow}& H^0(X,\mathcal O_X(-nm_sD+nm_ssF))\\
&&\uparrow&&\uparrow\\
0&\rightarrow &H^0(X,\mathcal O_X(nm_sH_s-nm_ssF))&\rightarrow&H^0(X,\mathcal O_X(-nm_sD))\\
&&\uparrow&&\uparrow\\
&&0&&0
\end{array}
$$

We have that 
$$
\alpha_n(H^0(X,\mathcal O_X(-nm_sD+nm_ssF))\cong H^0(X,\mathcal O_X(-nm_sD_s)
/H^0(X,\mathcal O_X(-nm_sD)).
$$
 We will  show that $\lambda_R( \alpha_n(H^0(X,\mathcal O_X(-nm_sD+nm_ssF)))$ grows like $n^d$.
To do this, it suffices to show that $\lambda_R(\beta_n(H^0(X,\mathcal O_X(nm_sH_s)))$ grows like $n^d$, since $\overline\tau_s^n$ is an injection and 
$\overline\tau_s^n\beta_n(H^0(X,\mathcal O_X(nm_sH_s))\subset \alpha_n(H^0(X,\mathcal O_X(-nm_sD+nm_ssF)))$.

Now $\beta_n$ is surjective for large $n$ since $H_s-sF$ is ample, so that
$H^1(X,\mathcal O_X(nm_sH_s-nm_ssF))=0$ for $n\gg0$. Thus it suffices to show that 
$\lambda_R(H^0(X,O_X(nm_sH_s)\otimes_{\mathcal O_{nm_ssF}}))$ grows of order $n^d$ for $n\gg 0$.

We have short exact sequences
\begin{equation}\label{eqR6}
0\rightarrow \mathcal O_X(nm_sH_s-jF)\otimes\mathcal O_F\rightarrow
 \mathcal O_X(nm_sH_s)\otimes\mathcal O_{(j+1)F}\rightarrow
 \mathcal O_X(nm_sH_s)\otimes\mathcal O_{jF}\rightarrow 0
 \end{equation}
 for $1\le j\le nm_ss-1$.

 For $0\le j\le nm_ss$,
$$
nm_sH_s-j F=\frac{nm_s(t-s)}{b_1}(\frac{1}{2}A +\frac{1}{2}(A-\frac{2b_1}{(t-s)}\frac{j}{nm_s}F))
$$
 is the sum of two ample divisors as shown in (\ref{eqR4}). This decomposition continues to hold when restricted to $F$. Thus by Fujita's vanishing theorem (\cite{Fuj} over all fields, also Theorem 1.4.35 page 66 \cite{La1} in characteristic zero) we have that there exists $n_0$ such that $n\ge n_0$ implies that $H^i(F,\mathcal O_X(nm_sH_s-j F)\otimes\mathcal O_F)=0$ for $i>0$ and $0\le j\le nm_ss$.

Let $h^i(\mathcal F)=\dim_kH^i(F,\mathcal F)=\lambda_R(H^i(F,\mathcal F))$ and $\chi(\mathcal F)=h^0(\mathcal F)-h^1(\mathcal F)$ if $\mathcal F$ is a coherent sheaf on $F$.

Restricting to $n\ge n_0$ and 
taking cohomology of the exact sequences (\ref{eqR6}), we obtain that
$$
\begin{array}{lll}
\lambda_R(H^0(X,\mathcal O_X(nm_sH_s)\otimes\mathcal O_{nm_ssF}))&=&
\sum_{j=0}^{nm_ss-1}h^0(\mathcal O_X(nm_sH_s-jF)\otimes\mathcal O_F)\\
& =&\sum_{j=0}^{nm_ss-1}\chi(\mathcal O_X(nm_sH_s-jF)\otimes\mathcal O_F).
\end{array}
$$
Now 
$\chi(\mathcal O_X(nm_sH_s-jF)\otimes\mathcal O_F)$ is  a polynomial in $n$ and $j$ of total degree $d-1=\dim F$ (the Snapper polynomial \cite{S}, \cite{K}, \cite{AG}).  The bi-homogeneous part of 
$\chi(\mathcal O_X(nm_sH_s-jF)\otimes\mathcal O_F)$ of total degree $d-1$ is 
$\frac{((nm_sH_s-jF)^{d-1}\cdot F)}{(d-1)!}$ (by a variation of Theorem 19.16 \cite{AG}). We can do these calculations after making a base change by an algebraic closure of $k$ since the intersection theory from the Snapper polynomial is valid over an arbitrary (not necessarily reduced) projective scheme.

Thus for large $n$,
$$
\lambda_R(H^0(X,\mathcal O_X(nm_sH_s)\otimes\mathcal O_{nm_sF}))
$$
is a polynomial $P_s(n)$ in $n$ of degree $\le d$, and 
$$
P_s(n)=\sum_{j=0}^{nm_ss-1} \frac{((nm_sH_s-jF)^{d-1}\cdot F)}{(d-1)!}
+\mbox{ terms in $n$ of degree $<d$.}
$$
Define 
$$
Q_s(n)=\sum_{j=0}^{ns-1}\frac{\left((nH_s-jF)^{d-1}\cdot F\right)}{(d-1)!}.
$$
$Q_s(n)$ is a polynomial in $n$ of degree $\le d$. We have that
$$
P_s(n)=Q_s(m_sn)+\mbox{ terms of degree $<d$}.
$$

\begin{equation}\label{eqR7}
\begin{array}{l}
(d-1)!Q_s(n)=\sum_{j=0}^{ns-1} ((nH_s-jF)^{d-1}\cdot F)\\
=((nH_s)^{d-1}\cdot F)+\sum_{j=1}^{ns-1}\left[\sum_{k=0}^{d-1}\binom{d-1}{k}(-1)^{d-k-1}(H_s^k\cdot F^{d-k})n^kj^{d-1-k}\right]\\
=((nH_s)^{d-1}\cdot F)+\sum_{k=0}^{d-1}\left[\binom{d-1}{k}(-1)^{d-k-1}(H_s^k\cdot F^{d-k})n^k\left(\sum_{j=1}^{ns-1}j^{d-1-k}\right)\right].
\end{array}
\end{equation}

By Faulhaber's formula, $\sum_{k=1}^nk^p$ is a polynomial in $n$ with leading term
$\frac{n^{p+1}}{p+1}$ (c.f. \cite{Bear}). Thus $\sum_{j=1}^{ns-1}j^{d-1-k}$ is a polynomial in 
$n$ of degree $d-k$, whose leading term is $\frac{s^{d-k}}{d-k}n^{d-k}$. Substituting into (\ref{eqR7}), we see that the coefficient $\sigma(s)$ of the  term of degree $d$ of the polynomial 
$Q_s(n)$ is
$$
\sigma(s)=\frac{1}{(d-1)!}\sum_{k=0}^{d-1}(-1)^{d-k-1}\binom{d-1}{k}\frac{s^{d-k}}{d-k}\left(\frac{t-s}{b_1}\right)^k(A^k\cdot F^{d-k})
$$
which is a polynomial in $s$. We have that
$$
\sigma(s)=s\frac{1}{(d-1)!}\left(\frac{t}{b_1}\right)^{d-1}(A^{d-1}\cdot F)+\mbox{ higher degree terms in $s$}.
$$
In particular, $\sigma(s)$ is a nonzero polynomial, and since $\frac{1}{(d-1)!}\left(\frac{t}{b_1}\right)^{d-1}(A^{d-1}\cdot F)>0$ (as $A$ is ample) if $s$ is sufficiently small within the region 
  $0<s\le\frac{t}{2b_1+1}$, we have that $\sigma(s)>0$. We now fix such an $s$. Since $Q_s(n)$ is a polynomial in $n$ of degree $d$, $P_s(n)$ is a polynomial in $n$ of degree $d$.

We have natural inclusions 
$$
H^0(X,\mathcal O(-nD_s))/H^0(X,\mathcal O_X(-nD))\subset 
H^0(X,\mathcal O_X(-n(\sum_{\pi(F_i)\ne m_R}a_iF_i)))
/H^0(X,\mathcal O_X(-nD)).
$$
 
Since $P_s(n)$ is a polynomial in $n$ of degree $d$ (with positive leading coefficient), there exists a constant $c>0$ such that for $n\gg 0$, we have that
$$
\begin{array}{lll}
0<c&\le& \frac{P_s(n)}{m_s^dn^d}=\frac{\lambda_R(H^0(nm_sF,\mathcal O_X(nm_sH_s)\otimes\mathcal O_{nm_sF}))}{m_s^dn^d}\\
&\le& \frac{\lambda_R (\alpha_n H^0(X,\mathcal O_X(-nm_sD+nm_sF)))}{m_s^dn^d}\\
&=& \frac{\lambda_R(H^0(X,\mathcal O_X(-nm_sD_s))/H^0(X,\mathcal O_X(-nm_sD))}{m_s^dn^d}\\
&\le &\frac{\lambda_R(H^0(X,\mathcal O_X(-nm_s(\sum_{\pi(F_i)\ne m_R}a_iF_i)))/
H^0(X,\mathcal O_X(-nm_sD))}{m_s^dn^d}.
\end{array}
$$
Thus 
$$
\begin{array}{lll}
\epsilon(\mathcal I)&=&d! \lim_{n\rightarrow \infty}\frac{\lambda_R(I(nD):m_R^{\infty}/I(nD))}{n^d}\\
&=& d!\lim_{n\rightarrow \infty}\frac{\lambda_R((H^0(X,\mathcal O_X(-n(\sum_{\pi(F_i)\ne m_R}a_iF_i)))/
H^0(X,\mathcal O_X(-nD))}{n^d}\\
&=& d!\lim_{n\rightarrow \infty}\frac{\lambda_R((H^0(X,\mathcal O_X(-nm_s(\sum_{\pi(F_i)\ne m_R}a_iF_i)))/
H^0(X,\mathcal O_X(-nm_sD))}{m_s^dn^d}\\
&\ge&d!c>0.
\end{array}
$$

\section{Epsilon multiplicity under inclusions of filtrations}
In this section we consider filtrations $\mathcal I,\mathcal J$ such that $\mathcal J\subset \mathcal I$ and discuss the existence of their epsilon multiplicities.
\begin{Proposition}\label{finite length}
Let $(R,\mm)$ be a Noetherian local ring of dimension $d>0$ and $\mathcal J=\{J_n\}, \mathcal I=\{I_n\}$ be filtrations of $R$ such that $J_n\subset I_n$ and $\lambda_R(I_n/J_n)<\infty$ for all $n\geq 1$. Suppose $\mathcal J$ satisfies $A(c)$ for some $c\in\ZZ_>0$.  Then the following hold.
\begin{enumerate}
\item [$(i)$] $\mathcal I$ satisfies $A(c)$ and $\epsilon(\mathcal I), \epsilon(\mathcal J)$ exist as limits.
\item [$(ii)$]  Suppose $R$ is analytically irreducible. If  $R[\mathcal I]$ is integral over $R[\mathcal J]$ then $\epsilon(\mathcal I)=\epsilon(\mathcal J)$.
\item [$(iii)$] The converse of $(ii)$ is not true in general.	
\end{enumerate}	
\begin{proof}
Since $\lambda_R(I_n/J_n)<\infty$, we have $J_n\subset I_n\subset J_n^{sat}$ for all $n\geq 1$. Hence 
$$J_n\cap \mm^{cn}=I_n\cap \mm^{cn}=J_n^{sat}\cap \mm^{cn}$$ for all $n\geq 1$.

$(i)$  Let $x\in I_n^{sat}\cap \mm^{cn}$ for any $n\geq 1$. Then $\mm^lx\in I_n\cap \mm^{cn}=J_n\cap \mm^{cn}\subset J_n$ for some $l\in\ZZ_{>0}$. Hence $x\in J_n^{sat}\cap \mm^{cn}=I_n\cap \mm^{cn}$. Therefore $\mathcal I$ satisfies $A(c)$ and by Theorem \ref{ThmE1}, $\epsilon(\mathcal I)$ and $\epsilon(\mathcal J)$ exist as limits and they are real numbers.

$(ii)$ By \cite[Theroem 6.1]{C1}, $\lim\limits_{n\to\infty}\lambda_R(I_n/J_n)/n^d$ exists. Now from  the following  short exact sequence 
$$0\longrightarrow I_n/J_n\longrightarrow R/J_n\longrightarrow R/I_n\longrightarrow 0$$ of $R$-modules, we get a short exact sequence 
$$0\longrightarrow H_{\mm }^0(I_n/J_n)=I_n/J_n\longrightarrow H _{\mm }^0(R/J_n)\longrightarrow H_{\mm}^0(R/I_n)\longrightarrow 0$$
 of local cohomology modules. Therefore 
 \begin{equation}\label{E}\epsilon(\mathcal J)=\epsilon(\mathcal I)+d!\lim\limits_{n\to \infty}\frac{\lambda_R(I_n/J_n)}{n^d}.\end{equation}

 If $R[\mathcal I]$ is integral over $R[\mathcal J]$ then by \cite[Theorem 1.5]{PS}, we have $\lim\limits_{n\to\infty}\lambda_R(I_n/J_n)/n^d=0.$ Therefore by equation (\ref{E}), we have $\epsilon(\mathcal I)=\epsilon(\mathcal J)$.

$(iii)$ See example \ref{Example6}. In this example, the filtrations $\mathcal I$ and $\mathcal J$ both satisfy $A(2)$, $\epsilon(\mathcal I)=\epsilon(\mathcal J)$ but $R[\mathcal I]$ is not  integral over $R[\mathcal J]$.
	\end{proof}	
\end{Proposition}

Using Theorem \ref{ep} and Proposition \ref{finite length}, we get the following.
\begin{Corollary}
	Let $R$ be a Noetherian local domain of dimension $d>0$ and $\mathcal J=\{J_n\}$ be a discrete valued filtration. Suppose $ \mathcal I=\{I_n\}$ is a filtration such that $J_n\subset I_n$ and $\lambda_R(I_n/J_n)<\infty$ for all $n\geq 1$. Then $\mathcal I$ satisfies $A(c)$ for some $c\in\ZZ_{>0}$ and $\epsilon(\mathcal I)$ exists as a limit.
\end{Corollary}	
\begin{Remark}
In Proposition \ref{finite length},	if we replace ``$\mathcal J$ satisfies $A(c)$" by ``$\mathcal I$ satisfies $A(c)$" then $\epsilon(\mathcal I)$ exists as a limit but $\epsilon (\mathcal J)$ may not exist. For example, let $\mathcal I=\{I_n=(x^2,xy^n)\}$ and $\mathcal J=\{J_n=(x^2,xy^{n^3})\}$ be filtrations in $k[x,y]_{(x,y)}$. Then $\mathcal I$ satisfies $A(2)$, $\epsilon(\mathcal I)=0$ and $\epsilon (\mathcal J)=\infty$.
	\end{Remark}
	
\begin{Remark}
	If we drop the condition that  $\mathcal J$ satisfies $A(c)$ for some $c\in\ZZ_>0$ in Proposition \ref{finite length}, then $(ii)$ of Proposition \ref{finite length} is not true in general (see Example \ref{Example 0}). 
\end{Remark}

The next example shows that the property of $A(c)$ does not  descend in integral extensions of filtrations. 

\begin{Example}\label{Example4}  There exists a filtration $\mathcal J$ which satisfies $A(2)$ with the following properties
\begin{enumerate}
\item[1)] Given $c\in \ZZ_{\ge 2}$, there exists a subfiltration $\mathcal K$ of $\mathcal J$ such that
$\overline{R[\mathcal K]}=\overline{R[\mathcal J]}$ and $c$ is the smallest positive integer such that $\mathcal K$ satisfies $A(c)$.
\item[2] There exists a subfiltration $\mathcal H$ of $\mathcal J$ such that $\overline{R[\mathcal H]}=\overline{R[\mathcal J]}$ but $\mathcal H$ does not satisfy $A(c)$ for any $c\in \ZZ_{>0}$.
\end{enumerate}
\end{Example}

We have that $\epsilon(\mathcal J)=\epsilon(\mathcal K)=0$ but $\epsilon(\mathcal H)=\frac{1}{2}$. Recall that we have defined the epsilon multiplicity of a filtration as a limsup in (\ref{epdef}).
We now construct the example.
Let $\mathcal J=\{J_n\}$ where $J_n$ is the ideal $(x^2,xy^n)$ in $R=k[x,y]_{(x,y)}$.
Let $\tau$ be any increasing function such that $\tau(n)\ge n$ for all $n$ and let $\mathcal I=\mathcal I_{\tau}=\{I_n\}$, where $I_n=(x^2,xy^{\tau(n)})$ as defined in Section \ref{SecEF}. Then $\mathcal I_{\tau}\subset \mathcal J$.

We will show that $\overline{R[\mathcal I_{\tau}]}=\overline{R[\mathcal J]}$. To establish this, we need only show that $xy^nt^n$ is integral over $R[I_{\tau}]=\sum_{n\ge 0}I_nt^n$. This follows since 
$$
(xy^nt^n)^2=x^2y^{2n}t^{2n}\in I_{2n}t^{2n}.
$$
%{\color{red}$$
%	(xy^nt^n)^2=x^2y^{2n}t^{2n}\in I_{2n}.
%	$$}

Let $\mathcal H$ be the filtration of Example \ref{Example1}. In Example \ref{Example1} we showed that 
$\mathcal H$ does not satisfy $A(c)$ for any $c$, so that the second statement of Example \ref{Example4} holds.

Let $a\in \ZZ_{>0}$ and let $\mathcal K=\{K_n\}$ where $K_n= (x^2, xy^{an})$.

 We will establish that $\mathcal K$ satisfies $A(c)$ if and only if $c>a$.

We have that $K_n:m_R^{\infty} =(x)$ for all $a$ and $n$. Thus for $c\in \ZZ_{>0}$,
$$
(K_n:m_R^{\infty})\cap m_R^{cn}=xm_R^{cn-1}
$$
and
$$
K_n\cap m_R^{cn}=\left\{
\begin{array}{ll}
m_R^{cn-2}x^2+xy^{an}m_R^{cn-an-1}=xm_R^{cn-1}&\mbox{ if $cn\ge an+1$}\\
m_R^{cn-2}x^2&\mbox{ if }cn<an+1.
\end{array}\right.
$$
Thus $\mathcal K$ satisfies $A(c)$ if and only if $c>a$. Taking $a=c-1$, we obtain the conclusions of the first statement of    Example \ref{Example4}.

We also have that the property of $A(c)$ does not ascend under inclusions of filtrations. 

\begin{Example}\label{Example5} There exists an inclusion of filtrations  $\mathcal I\subset \mathcal J$ such that $\mathcal I$ satisfies $A(1)$ but $\mathcal J$ does not satisfy $A(1)$.
\end{Example}

The example is constructed as follows. Let $R=k[x,y]_{(x,y)}$. Let $\mathcal I=\{I_n\}$ where $I_n=(x^{3n})$ and $\mathcal J=\{J_n\}$ where $J_n=(x^nm_R^{2n})$. We have that $\mathcal I$ satisfies $A(1)$.
For all $n$, $J_n:m_R^{\infty}=(x^n)$ so that $(J_n:m_R^{\infty})\cap m_R^n=(x^n)$ and $J_n\cap m_R^n=J_n$ so $\mathcal J$ does not satisfy $A(1)$.

The above examples shows that it is not easy to approximate filtrations which satisfy $A(c)$ by filtrations which also satisfy  $A(c)$. In the case that we can, we may compute the epsilon multiplicity of a filtration as a limit of the epsilon multiplicities of the approximating filtrations. The following proposition is proven by an extension of the proof of \cite[Theorem 6.1]{C1}.

\begin{Proposition} \label{PropE4} Suppose that $R$ is an analytically unramified local ring 
and $\mathcal I$ is a filtration of $R$ which satisfies $A(c)$ for some $c\in \ZZ_{>0}$. Let $\mathcal I_i=\{I[i]_n\}$ be subfiltrations of $\mathcal I$ for $i\in \ZZ_{>0}$ such that $\mathcal I_j\subset \mathcal I_i$ if $j\ge i$, $\cup_{i=1}^{\infty}=\mathcal I$ and $\mathcal I_i$ satisfy $A(c)$ for all $i$. Then
$$
\lim_{i\rightarrow \infty}\epsilon(\mathcal I[i])=\epsilon(\mathcal I).
$$
\end{Proposition}

\section{Integral closure, multiplicity  and epsilon multiplicity of filtrations}\label{closure}
In this section we study the relationship between epsilon multiplicities of two filtrations and their integral closures. The following examples show that Theorem \ref{PropE6} does not extend to general filtrations.
	\begin{Example}\label{Example6}
		We have filtrations $\mathcal J\subset \mathcal I$ in a Noetherian local ring $R$ such that $\epsilon(\mathcal I_\mfp)=\epsilon(\mathcal J_\mfp)$ for all $\mfp\in\Spec R$ but $\overline{R[\mathcal I]}\neq\overline {R[\mathcal J]}$.
	\end{Example}
	In the example, $R=k[x,y]_{(x,y)}$. Further,  $\mathcal I$ and $\mathcal J$ satisfy $A(2)$.

We now construct the example. 	Consider the filtrations $\mathcal I=\{I_n=(x^n)\}$ and $\mathcal J=\{J_n=(x^{n+1},x^{n}y)\}$ in $R=\CC[x,y]_{(x,y)}$. Then $V(I_1)=V(J_1)=\{P=(x),Q=(x,y)\}$ and $\epsilon(\mathcal I_\mfp)=\epsilon(\mathcal J_\mfp)=0$ for all $\mfp\in\Spec R\setminus V(I_1)$. Note that $I_nR_P=J_nR_P=(x^n)R_P$ for all $n\in\NN$. Hence $\epsilon(\mathcal I_P)=\epsilon(\mathcal J_P)$. Since $I_n$ does not have $Q$ as an associated prime for all $n\in\NN$, we have $\epsilon(\mathcal I)=0$. For all $n\geq 1$, we have 
$$
(J_n:m_R^{\infty})/J_n=(x^n)/(x^{n+1},x^ny)\cong R/(x,y)
$$
so $\lambda_R((J_n:m_R^{\infty})/J_n)=1$. 
Thus $\epsilon(\mathcal J)=\lim\limits_{n\to\infty}2/n^2=0$.
\\Now if $\overline{R[\mathcal I]}=\overline {R[\mathcal J]}$, since $R\subset R[\mathcal J]\subset \overline {R[\mathcal J]}$ and  $\overline{R[\mathcal I]}$ is a finitely generated $R$-algebra, by the Artin-Tate lemma, we have $R[\mathcal J]$ is a finitely generated $R$-algebra which is a contradiction. Hence $\overline{R[\mathcal I]}\neq\overline {R[\mathcal J]}$.

	\begin{Example}\label{Example 0}
	We have filtrations $\mathcal J\subset \mathcal I$ in a Noetherian local ring $R$ such that $\overline{R[\mathcal I]}=\overline {R[\mathcal J]}$ but $\epsilon(\mathcal I)\neq\epsilon(\mathcal J)$.
	
	Consider the filtrations $\mathcal I=\{I_n=(x^2,xy^n)\}$ and $\mathcal J=\{J_n=(x^2,xy^{n^2})\}$ in $R=\CC[x,y]_{(x,y)}$. Since $(xy^nt^n)^2\in J_{2n}t^{2n}$, we have $\overline{R[\mathcal I]}=\overline {R[\mathcal J]}$. Note that $$\epsilon(\mathcal I)=\lim\limits_{n\to\infty}2!n/n^2=0\neq 2 =\lim\limits_{n\to\infty}2!n^2/n^2=\epsilon(\mathcal J).$$ 
\end{Example}

Let $\mfa$ be an $\mm$-primary ideal of a local ring $(R,\mm)$ and $N$ be a finitely generated $R$-module with $\dim N=r.$ Define $$e_{\mfa}(N)=\lim_{k\to\infty}\frac{l_R(N/\mfa ^kN)}{k^r/r!}.$$ If $s\geq r=\dim N,$ define (\cite[V.2]{Se}, \cite[4.7]{BH})
\[ e_s({\mfa},N)= \left\{
\begin{array}{l l}
e_{\mfa}(N) & \quad \text{if $\dim N=s$ }\\
0 & \quad \text{if $\dim N<s.$ }
\end{array} \right.\] 
Let  $(R,\mm)$ be a  local ring and $\mathcal I=\{I_n\}$ be a filtration of ideals of $R$.  In \cite[Definition 3.2]{CS}, we defined the dimension of the filtration $\mathcal I$ to be  $s(\mathcal I)=\dim R/I_n$ (for any $n\geq 1$). If $R$ is an analytically unramified local ring, $N$ is a finitely generated $R$-module, $\mathfrak a$ is an $\mm$-primary ideal  and $\mathcal I$ is a filtration on $R$ then by \cite[Proposition 4.2]{CS}, for $s\in \NN$ with $s(\mathcal I)\le s\le d$, we have 
$$e_s(\mathfrak a,\mathcal I;N):=\lim_{m\rightarrow\infty}\frac{e_s(\mathfrak a,N/I_mN)}{m^{d-s}/(d-s)!}
$$exists and 
\begin{equation}\label{two}
e_s(\mathfrak a,\mathcal I;N)=\sum_{\mathfrak p}e_{R_{\mathfrak p}}(\mathcal I_{\mathfrak p},N_{\mfp})e_{\mathfrak a}(R/\mathfrak p)
\end{equation}
where the sum is over all $\mathfrak p\in \Spec R$ such that $\dim R/\mathfrak p=s$ and $\dim R_{\mathfrak p}=d-s$. We write $e_s(\mathcal I)=e_s(\mm,\mathcal I;R)$.	
\begin{Remark}
	Note that for any filtration $\mathcal I$ in a local ring $R$ with $\dim N(\hat{R})<\dim R$, by \cite[Theorem 1.1]{C1},  we have $\epsilon (\mathcal I_\mfp)=e_{R_\mfp}(\mathcal I_\mfp)$ exists for all $\mfp\in \MinAss(R/I_n)$ and $n\geq 1$ .	
\end{Remark}
\begin{Proposition}\label{epeq}
	Let $(R,\mm)$ be an analytically unramified local ring and $\mathcal I, \mathcal J$ be filtrations in $R$ with $\mathcal J\subset \mathcal I$  and $\dim R/\mathcal J=s$. Then  $e_s(\mathcal I)=e_s(\mathcal J)$ if and only if 	$\epsilon(\mathcal I_\mfp)=\epsilon(\mathcal J_\mfp)$ for all $\mfp\in\Spec R$ with $\dim R/\mfp=s$ and $\dim R_\mfp=\dim R-s$.
\end{Proposition}
\begin{proof}
Let $V=\{\mfp\in\Spec R: \dim R/\mfp=s\mbox{ and }\dim R_\mfp=\dim R-s\}$ and for all $n\geq 1$, $V\cap \MinAss(R/J_1)=V\cap \MinAss(R/J_n)=\{\mfp_1,\ldots,\mfp_r\}$. 

Let $\mfp\in V\setminus \{\mfp_1,\ldots,\mfp_r\}$. Then $J_nR_\mfp=R_\mfp$ for all $n\geq 1$. Since $\dim R/\mathcal I\leq \dim R/\mathcal J=s$, we also have $I_nR_\mfp=R_\mfp$ for all $n\geq 1$. Hence $\epsilon(\mathcal I_\mfp)=\epsilon(\mathcal J_\mfp)=e_{R_\mfp}(\mathcal I_{\mfp})=e_{R_\mfp}(\mathcal J_{\mfp})=0$.  

Let $\mfp\in  \{\mfp_1,\ldots,\mfp_r\}$. Then $\epsilon(\mathcal J_{\mfp})=e_{R_{\mathfrak p}}(\mathcal J_{\mathfrak p})$. Since $\dim R/\mathcal I\leq \dim R/\mathcal J=s$, we have  $\epsilon(\mathcal I_{\mfp})=e_{R_\mfp}(\mathcal I_{\mfp})$. 

Thus  by taking $N=R$ and $\mathfrak a=\mm$ in equation (\ref{two}), we have 
\begin{eqnarray*}
	e_s(\mathcal I)=e_s(\mathcal J)
	&\Leftrightarrow& \sum_{\mfp\in V}e_{R_{\mathfrak p}}(\mathcal I_{\mathfrak p})e_{\mm}(R/\mathfrak p)=\sum_{\mfp\in V}e_{R_{\mathfrak p}}(\mathcal J_{\mathfrak p})e_{\mm}(R/\mathfrak p)\\
	&\Leftrightarrow& \sum_{i=1}^r[e_{R_{\mathfrak p_i}}(\mathcal J_{\mathfrak p_i})-e_{R_{\mathfrak p_i}}(\mathcal I_{\mathfrak p_i})]e_{\mm}(R/\mathfrak p_i)=0\\
	&\Leftrightarrow&e_{R_{\mathfrak p_i}}(\mathcal I_{\mathfrak p_i})=e_{R_{\mathfrak p_i}}(\mathcal J_{\mathfrak p_i})\mbox { for all }i=1,
	\ldots,r.\\
	&\Leftrightarrow&\epsilon(\mathcal I_{\mfp_i})=\epsilon(\mathcal J_{\mfp_i})\mbox { for all }i=1,
	\ldots,r.
\end{eqnarray*}
\end{proof}	

\begin{Remark}\label{nceq}
Let  $(R,\mm)$ be a Noetherian local ring and $\mathcal I,\mathcal J$ be filtrations of $R$ such that $\mathcal J\subset \mathcal I$. Suppose $R[\mathcal J]=\overline{R[\mathcal J]}=\overline {R[\mathcal I]}$. Then $\mathcal I=\mathcal J$.
\end{Remark}

We recall a few definitions from \cite{CS}.	

An $s$-divisorial filtration is a divisorial filtration $\mathcal J=\{J_m=I(\nu_1)_{\lceil ma_1\rceil}\cap\cdots\cap I(\nu_r)_{\lceil ma_r\rceil}\}$ with $\dim R/m_{\nu_i}\cap R=s$ for all $i=1,\ldots,r$.  

A filtration $\mathcal J$ is called a bounded filtration if there exists a divisorial filtration
$$
\mathcal C=\{C_m=I(\nu_1)_{\lceil ma_1\rceil}\cap\cdots\cap I(\nu_r)_{\lceil ma_r\rceil}\}
$$
such that $\overline{R[\mathcal J]}=R[\mathcal C]$.

A filtration $\mathcal I$ is called a bounded $s$-filtration if there exists an $s$-divisorial filtration 
$\mathcal C$ such that $\overline{R[\mathcal J]}=R[\mathcal C]$.

We now prove Theorem \ref{equality} from the introduction, which we restate here for the convenience of the reader. 

\begin{Theorem}
	Let $(R,\mm)$ be an excellent local domain and $\mathcal I$ be a filtration of ideals in $R$. Then the following hold.
	\begin{enumerate}
		\item[{(1)}] Suppose $\mathcal J$ is an $s$-divisorial filtration such 
		that $\mathcal J\subset \mathcal I$. Then the following are equivalent.
		\begin{enumerate}
			\item[{(i)}] $\epsilon(\mathcal I_\mfp)=\epsilon(\mathcal J_\mfp)$ for all $\mfp\in\Spec R$. 
			\item[{(ii)}] $\epsilon(\mathcal I_\mfp)=\epsilon(\mathcal J_\mfp)$ for all $\mfp\in\Spec R$ such that $\ell(\mathcal J_\mfp)=\dim R_\mfp$. 
			\item[{(iii)}] $\epsilon(\mathcal I_\mfp)=\epsilon(\mathcal J_\mfp)$ for all $\mfp\in\Spec R$ with $\dim R/\mfp=s$.
			\item[{(iv)}] $\mathcal I=\mathcal J$.
			\item[{(v)}] $\overline{R[\mathcal I]}=\overline {R[\mathcal J]}$.
		\end{enumerate}
		\item[{(2)}] Suppose $\mathcal J$
		is a bounded $s$-filtration such 
		that $\mathcal J\subset \mathcal I$. Then $\overline{R[\mathcal I]}=\overline {R[\mathcal J]}$  if and only if $\epsilon(\mathcal I_\mfp)=\epsilon(\mathcal J_\mfp)$ for all $\mfp\in\Spec R$ with $\dim R/\mfp=s$.
	\end{enumerate}
\end{Theorem}
\begin{proof}
	Note that $\dim R/\mathcal I\leq \dim R/\mathcal J=s$. Suppose that $s=0$. Then $e_R(\mathcal I)=\epsilon(\mathcal I), e_R(\mathcal J)=\epsilon(\mathcal J)$. 	Thus $(1)$ and $(2)$ follow from \cite[Theorem 1.4]{C} , \cite[Theorem 13.1]{C} and Remark \ref{nceq}. Therefore we assume $s>0$.
	
$(1)$	It is clear that $(i)\Rightarrow (ii)$, $(iv)\Rightarrow (i)$ and $(iv)\Rightarrow (v)$. By Lemma \ref{integrallyclosed} and Remark \ref{nceq}, we have $(v)\Rightarrow (iv)$.

Next we show that $(ii)$ implies $(iii)$. For any $\mfp\in\Spec R$
with $\dim R/\mfp=s$, we have either $\mfp\notin V(J_1)$ or $\mathcal J_\mfp$ is a filtration of $\mfp R_\mfp$-primary ideals. If $\mfp\notin V(J_1)$ then $\mfp\notin V(I_1)$. Hence $\epsilon(\mathcal I_\mfp)=\epsilon(\mathcal J_\mfp)=0$. If $\mathcal J_\mfp$ is a filtration of $\mfp R_\mfp$-primary ideals then $\ell(\mathcal J_\mfp)=\dim R_\mfp$ and by hypothesis $\epsilon(\mathcal I_\mfp)=\epsilon(\mathcal J_\mfp)$.

Now we prove that $(iii)$ implies $(iv)$. By Proposition \ref{epeq}, we have $e_s(\mathcal I)=e_s(\mathcal J)$. Therefore by \cite[Theorem 6.7]{CS}, we get $\mathcal I=\mathcal J$.

	$(2)$ Suppose $\overline{R[\mathcal I]}=\overline {R[\mathcal J]}$. Then by \cite[Theorem 5.1]{CS}, we have $e_s(\mathcal I)=e_s(\mathcal J)$ and hence by Proposition \ref{epeq}, we have $\epsilon(\mathcal I_\mfp)=\epsilon(\mathcal J_\mfp)$ for all $\mfp\in\Spec R$ with $\dim R/\mfp=s$.

Suppose $\epsilon(\mathcal I_\mfp)=\epsilon(\mathcal J_\mfp)$ for all $\mfp\in\Spec R$ with $\dim R/\mfp=s$.
Then by Proposition \ref{epeq}, we have $e_s(\mathcal I)=e_s(\mathcal J)$. Using \cite[Theorem 7.4]{CS}, we get the required result.
	
\end{proof}	

\begin{Corollary}
		Let $(R,\mm)$ be an excellent local domain and $\mathcal I$ be a filtration of ideals in $R$. Let $\mathcal J=\{J^n\}$ where $J$ is an equimultiple ideal such that $J^n\subset I_n$ for all $n\geq 1$. 
Then the following are equivalent.
		\begin{enumerate}
			\item[{(i)}] $\overline{R[\mathcal I]}=\overline {R[\mathcal J]}$.
			\item[{(ii)}]  $\epsilon(\mathcal I_\mfp)=\epsilon(\mathcal J_\mfp)$ for all $\mfp\in\Spec R$ with $\dim R/\mfp=\dim R-\ell(J)$.
			\item[{(iii)}]$\epsilon(\mathcal I_\mfp)=\epsilon(\mathcal J_\mfp)$ for all $\mfp\in\Spec R$ such that $\ell(\mathcal J_\mfp)=\dim R_\mfp$. 
		\end{enumerate}
	\begin{proof}
		If $\height J=\dim R$ then the result follows from \cite[Theorem 1.2]{C}. Suppose $\height J<\dim R$.
		Since $J$ is an equimultiple ideal, by \cite[Corollary 8.3]{CS},  $\mathcal J$ is a bounded $s$-filtration where $s=\dim R-\ell(J)$. Thus the equivalence of $(i)$ and $(ii)$ follows from Theorem \ref{equality} (2). 
\\ Since  $J$ is an equimultiple ideal, by \cite[Corollary 9, 9.3]{Ratliff}, all prime divisors of $\overline{J^n}$ are of height $\ell(J)$. Therefore for all $n\geq 1$ and any $\mfp\in\Ass(R/\overline{J^n})=\MinAss(R/J),$ we have $\dim R/\mfp=s$. Now using \cite[Corollary 4]{Ratliff}, $\ell (\mathcal J_\mfp)=\dim R_\mfp$ if and only if $\mfp\in\Ass(R/\overline{J^n})=\MinAss(R/J)$ and $\dim R/\mfp=s$. Thus $(ii)$ implies $(iii)$.
\\Now suppose $\mfp\in\Spec R$ such that $\dim R/\mfp=\dim R-\ell(J)=\dim R-\height J$.
If $\mfp\notin V(J)$ then $J_\mfp=R_\mfp=I_nR_\mfp$ and hence $\epsilon(\mathcal I_\mfp)=\epsilon(\mathcal J_\mfp)=0$. Suppose $\mfp\in V(J)$. Then $\mfp=\in\Ass(R/\overline{J^n})=\MinAss(R/J)$ and hence $\ell(\mathcal J_\mfp)=\dim R_\mfp$. Thus $(iii)$ implies $(ii)$.
		\end{proof}

\end{Corollary}	
%{\color{red} Are the statements in Theorem \ref{equality}(2)  equivalent to the following:
%\\$\epsilon(\mathcal I_\mfp)=\epsilon(\mathcal J_\mfp)$ for all $\mfp\in\Spec R$ such that $\ell(\mathcal J_\mfp)=\dim R_\mfp$ where $\mathcal J$ is a bounded $s$-filtration.
%}

\begin{Example}	
	Theorem \ref{equality} does not extend to be true for general divisorial filtrations. Let $\mathcal J=\{\overline{\mfp^n}\}$ and $\mathcal I=\{\mfp^{(n)}\}$ where $\mfp$ is a height 2 prime ideal in a regular local ring $R$ of dimension $3$ such that $\mathcal I$ is not finitely generated  $R$-algebra. Hence $\overline{R[\mathcal I]}\neq\overline{R[\mathcal J]}$.
	We have $\dim R/\mathcal J=1$. Note that 
	$\epsilon(\mathcal I_{\mfp})=e_1(\mathcal I)=e_1(\mathcal J)=\epsilon(\mathcal J_{\mfp})$ and $0=\epsilon(\mathcal I)\neq \epsilon(\mathcal J)=\epsilon(\mfp)>0$  by Proposition 2.1 (b) \cite{JMV}	, which shows that $\epsilon(I)=\epsilon(\{\overline{I^n}\})$ and Theorem \ref{TheoremI2} as $\ell(\mfp)=\dim R$.  Statement (iii)  of Theorem \ref{equality} is true but statements (i), (ii) (iv) and (v) are not.
\end{Example}

\begin{Example}
	Theorem \ref{equality} does not extend for bounded filtrations. Let $\mathcal J=\{\mfp^n\}$ and $\mathcal I=\{\mfp^{(n)}\}$ where $\mfp$ is a height 2 prime ideal in a regular local ring $R$ of dimension $3$ such that $\mathcal I$ is not finitely generated $R$-algebra. Hence $\overline{R[\mathcal I]}\neq\overline{R[\mathcal J]}$.
	We have $\dim R/\mathcal J=1$. Note that 
	$\epsilon(\mathcal I_{\mfp})=e_1(\mathcal I)=e_1(\mathcal J)=\epsilon(\mathcal J_{\mfp})$ and $0=\epsilon(\mathcal I)\neq \epsilon(\mathcal J)=\epsilon(\mfp)>0$ as $\ell(\mathcal J)=\dim R$.
\end{Example}

\end{document}